\documentclass[a4paper]{amsart}
\usepackage{amsmath,amssymb,amsthm,amscd}
\usepackage[all]{xy}

\author{Jan \v S\v tov\'\i\v cek}
\address{Charles University, Faculty of Mathematics and Physics, Department of Algebra \\ 
Sokolovsk\'{a} 83, 186 75 Prague 8, Czech Republic}

\author{Otto Kerner}
\address{Mathematisches Institut, Heinrich--Heine--Universit\" at D\" usseldorf, 
Universit\" atsstr.1, 40225 D\" usseldorf, Germany}

\author{Jan Trlifaj}
\address{Charles University, Faculty of Mathematics and Physics, Department of Algebra \\ 
Sokolovsk\'{a} 83, 186 75 Prague 8, Czech Republic}

\thanks{This research was supported by MSM 0021620839 and by the bilateral university exchange program of Charles University, Prague, and the Heinrich Heine University, D\" usseldorf. The first author was supported by the Research Council of Norway through Storforsk--project Homological and geometric methods in algebra.}

\title[Tilting via torsion pairs and almost hereditary noetherian rings]
{Tilting via torsion pairs and almost hereditary noetherian rings}
\date{\today}

\renewcommand{\iff}{if and only if }
\newcommand{\st}{such that }
\newcommand{\Z}{\mathbb{Z}}

\DeclareMathOperator{\Hom}{Hom}
\DeclareMathOperator{\End}{End}
\DeclareMathOperator{\Ext}{Ext}

\DeclareMathOperator{\Ker}{Ker}
\DeclareMathOperator{\Img}{Im}
\DeclareMathOperator{\Coker}{Coker}
\DeclareMathOperator{\pd}{pd}
\DeclareMathOperator{\id}{id}
\DeclareMathOperator{\gldim}{gl.dim.}
\DeclareMathOperator{\rgldim}{r.gl.dim.}
\DeclareMathOperator{\depth}{depth}
\newcommand{\inv}{^{-1}}

\newcommand{\cpx}[1]{{#1}^\cdot}
\newcommand{\RHom}{\mathbf{R}\!\Hom}

\newcommand{\A}{\mathcal{A}}
\newcommand{\B}{\mathcal{B}}
\newcommand{\C}{\mathcal{C}}
\newcommand{\D}{\mathcal{D}}
\newcommand{\E}{\mathcal{E}}
\newcommand{\F}{\mathcal{F}}
\newcommand{\clH}{\mathcal{H}}
\newcommand{\clS}{\mathcal{S}}
\newcommand{\T}{\mathcal{T}}
\newcommand{\X}{\mathcal{X}}
\newcommand{\Y}{\mathcal{Y}}

\newcommand{\ModR}{\mathrm{Mod}\textrm{-}R}
\newcommand{\modR}{\mathrm{mod}\textrm{-}R}
\newcommand{\indR}{\mathrm{ind}\textrm{-}R}
\newcommand{\ModS}{\mathrm{Mod}\textrm{-}S}
\newcommand{\modS}{\mathrm{mod}\textrm{-}S}
\newcommand{\radS}{\mathrm{rad}\textrm{-}S}
\newcommand{\modZ}{\mathrm{mod}\textrm{-}\Z}

\newcommand{\add}{\mathrm{add}}
\newcommand{\gen}{\mathrm{gen}}

\newcommand{\Db}[1]{D^b({#1})}
\newcommand{\Kb}[1]{K^b({#1})}
\newcommand{\Kbac}[1]{K^b_{ac}({#1})}

\newcommand{\Dl}[1]{D^{\leq{#1}}}
\newcommand{\Dr}[1]{D^{\geq{#1}}}

\theoremstyle{plain}
\newtheorem*{thm:main}{Main Theorem}

\newtheorem{thm}{Theorem}[section]
\newtheorem{prop}[thm]{Proposition}
\newtheorem{lem}[thm]{Lemma}
\newtheorem{cor}[thm]{Corollary}

\theoremstyle{definition}
\newtheorem{defn}[thm]{Definition}
\newtheorem{constr}[thm]{Construction}
\theoremstyle{remark}
\newtheorem{rem}[thm]{Remark}
\newtheorem{expl}[thm]{Example}

\begin{document}
\begin{abstract}
We generalize the tilting process by Happel, Reiten and Smal\o{} to the setting of finitely presented modules over right coherent rings. Moreover, we extend the characterization of quasi--tilted artin algebras as the almost hereditary ones to all right noetherian rings. 
\end{abstract}

\maketitle

\section{Introduction}
\label{sec:intro}

Classical tilting theory originated in the 1970s and concerned finitely generated $1$--tilting modules over artin algebras. Since then, many powerful generalizations have been developed. However, these are mainly restricted to artin algebras and categories with finite dimensional $\Ext$--spaces over a field, or they work with categories of all infinitely generated modules and are more of theoretical interest. In this paper we develop a computationally feasible method for working with derived equivalences of abelian categories. We apply it to extend the descriptions of quasi--tilted algebras by 
Happel, Reiten and Smal\o{} \cite{HRS} to the more general setting of right coherent and right noetherian rings.

\medskip

The 1996 Memoir \cite{HRS} provided a major extension of classical tilting theory, developing tilting theory with respect to 
a~tilting torsion pair in a locally finite hereditary abelian category. 
In particular the equivalence of the following three conditions was proved in  \cite{HRS} for each artin algebra~$R$:

\begin{enumerate}
\item[(i)] $R$ is quasi--tilted, that is, isomorphic to the endomorphism algebra of a tilting object in a locally finite hereditary abelian category. 
\item[(ii)] There is a split torsion pair in $\modR$ whose torsion--free class $\mathcal Y$ consists of modules of projective dimension $\leq 1$, and $R \in \mathcal Y$;
\item[(iii)] $R$ is almost hereditary, that is, $R$ has right global dimension $\leq 2$, and $\pd M \leq 1$ or $\id M \leq 1$ for each finitely generated indecomposable module $M$.
\end{enumerate}

In 2007, Colpi, Fuller, and Gregorio considered analogs of (i)--(iii) for arbitrary modules over arbitrary rings. In \cite{CF}, a version of the equivalence between (i) and (ii) was proved for $\ModR$, the category of all modules and tilting objects in hereditary cocomplete abelian categories. The exact relation of (ii) and (iii) in this setting remains, however, an open problem.

Colpi, Fuller, and Gregorio also suggested to consider the equivalence of (ii) and (iii) in the form stated above, but for arbitrary right noetherian rings $R$. They proved several results in this direction (see Section 5 for more details), but the equivalence remained an open problem. 

Here we give a short proof for the equivalence between (ii) and (iii) for all right noetherian rings. The main result of the paper is then 

\begin{thm:main}
The following are equivalent for a right noetherian ring $R$:
\begin{enumerate}
\item[(i)] $R$ is isomorphic to the endomorphism ring of a tilting object in a small hereditary abelian category.
\item[(ii)] There is a split torsion pair in $\modR$ whose torsion--free class $\mathcal Y$ consists of modules of projective dimension $\leq 1$, and $R \in \mathcal Y$;
\item[(iii)] $R$ is almost hereditary, i.e., $R$ has right global dimension $\leq 2$, and $\pd M \leq 1$ or $\id M \leq 1$ for each finitely generated indecomposable module $M$;
\end{enumerate}
Moreover, \emph{(i)} is equivalent to \emph{(ii)} for any right coherent ring $R$.
\end{thm:main}

Here we call an object $T$ in a small abelian category $\A$ tilting if it has projective dimension at most $1$, has no self--extensions, for each $X \in \A$, $\Hom_\A(T,X) = 0 = \Ext^1_\A(T,X)$ implies $X = 0$, and both $\Hom_\A(T,X)$ and $\Ext^1_\A(T,X)$ are finitely generated $\End_\A(T)$--modules. 

The work of Happel, Reiten and Smal\o{}~\cite{HRS} was motivated by obtaining a unified treatment for tilted and canonical artin algebras. Our results show that one can extend this framework to encompass further examples, for instance the class of serially tilted rings~\cite{CbF}. Moreover, the proofs of the key statements are quite short.

\medskip

Our paper is organized as follows. After recalling preliminary facts, we present a general theory for tilting in abelian categories using torsion pairs in Sections~\ref{sec:tilting}--\ref{sec:tilt_hered}. The definition and properties of tilting objects are given in Section~\ref{sec:tilt_obj}. 
In Section~\ref{sec:q_tilt}, we complete the proof of the Main Theorem. Finally, we illustrate it on a couple of examples in Section~\ref{sec:expl}. 

{\it Acknowledgment.} We wish to thank Idun Reiten, Dieter Happel, and Bernhard Keller for valuable comments on the results presented here.  

\section{Preliminaries}
\label{sec:prelim}

In what follows all rings are associative with unit, but not necessarily commutative. For a ring $R$, we denote by $\ModR$ the category of all (right $R$--) modules, by $\modR$ its subcategory consisting of all finitely presented modules, and by $\indR$ the subcategory of $\modR$ consisting of all indecomposable modules. Recall that a~ring $R$ is \emph{right coherent} if every finitely generated right ideal of $R$ is finitely presented. It is well known that $R$ is right coherent if and only if the category $\modR$ is abelian. For example, any right noetherian or right artinian ring is right coherent.

Let $\A$ be an abelian category. Although $\A$ may not have enough projectives or injectives, one can still define the \emph{projective dimension} of $X \in \A$ as $\pd_\A X = n$ where $n \geq 0$ is the minimal $m$ \st $\Ext^{m+1}_\B(X,-) \equiv 0$ or $n = \infty$ if no such $m$ exists. Dually, we define the \emph{injective dimension} of $X \in \A$. The \emph{global dimension} of $\A$ is defined by $\gldim \A = \sup \{\pd_\A X \mid X \in \A\}$, and $\A$ is said to be \emph{hereditary} if $\gldim\A \leq 1$. These concepts have the usual properties that are well known for module categories. In particular, $\gldim\A = n < \infty$ \iff $\Ext^{n+1}_\A(-,-) \equiv 0$ \iff $\Ext^i_\A(-,-) \equiv 0$ for each $i \geq n+1$.

Following the convention in~\cite{H}, we denote by $\Kb\A$ the category of bounded complexes over $\A$ modulo the ideal of null--homotopic chain complex morphisms. This is well known to be a triangulated category where the triangles are formed using mapping cones. By $\Db\A$, we denote the \emph{bounded derived category} of $\A$, that is, the localization of $\Kb\A$ with respect to the class $\Sigma$ of all quasi--isomorphisms.

The idea of localizing triangulated categories and constructing derived categories, studied by Verdier~\cite{V} in 1960's, is, nevertheless, much more general. A detailed account on this is given in~\cite[\S 2.1]{N}. A nice overview can also be found in~\cite{Kr}, for example. Let $\T$ be a triangulated category and $\clS \subseteq \T$ a triangulated subcategory. Denote by $\Sigma$ the class of all morphisms $X \to Y$ in $\T$ which can be completed to a triangle $X \to Y \to S \to X[1]$ \st $S \in \clS$. Then we can form a Verdier quotient $\T/\clS$ described as follows:
\begin{enumerate}
\item The objects of $\T/\clS$ coincide with the objects of $\T$.
\item The morphisms from $X$ to $Y$ are left fractions $X \overset{f}\to Z \overset{\sigma}\leftarrow Y$ (denoted $\sigma\inv f$ for short) \st $f \in \Hom_\T(X,Z)$ and $\sigma \in \Sigma$, modulo the following equivalence relation: $\sigma_1\inv f_1$ and $\sigma_2\inv f_2$ are equivalent if one can form a commutative diagram \st $\sigma \in \Sigma$:
$$
\xymatrix{
& Z_1 \ar[d] &						\\
X \ar[ur]^{f_1} \ar[dr]_{f_2} \ar[r]^f &
Z &
Y \ar[ul]_{\sigma_1} \ar[dl]^{\sigma_2} \ar[l]_\sigma	\\
& Z_2 \ar[u] &						\\
}
$$
\end{enumerate}
Equivalently, morphisms in $\T/\clS$ can be expressed as right fractions $f \sigma\inv$. The way to compose and add fractions is well known but somewhat technical, we refer for example to~\cite[\S 2.1]{N}. As with the usual Ore localization, we have $\sigma\inv f = 0$ in $\T/\clS$ \iff $\tau f = 0$ in $\T$ for some $\tau \in \Sigma$. If, moreover, $\clS$ is a thick subcategory of $\T$ (that is, triangulated and closed under those direct summands which exist in $\T$), then $\sigma\inv f$ is invertible \iff $f \in \Sigma$, \cite[2.1.35]{N}.

The category $\T/\clS$ inherits a natural triangulated structure from $\T$ \st the localization functor $Q: \T \to \T/\clS$ which sends $f: X \to Y$ to $1_Y\inv f$ is exact. However, $Q$ is neither full nor faithfull in general. The construction of the derived category fits into this framework: $\Db\A = \Kb\A / \Kbac\A$ where $\Kbac\A$ is the full subcategory given by all acyclic complexes.

The only limitation for this construction is in the possible set--theoretic problems arising out of the fact that there is a priori no reason why the collection of morphisms between given two objects of $\T/\clS$ should form a set and not a proper class~\cite[2.2]{N}. In many cases, it is obvious or well known that $\Hom_{\T/\clS}(X,Y)$ is always a set. In the case of the derived category $\Db\A$ of an abelian category $\A$, one knows that the $\Hom$--spaces are sets if
\begin{enumerate}
\item $\A$ is skeletally small,
\item $\A$ has enough projectives or enough injectives (in particular if $\A = \ModR$).
\end{enumerate}
Although naturally occuring abelian categories typically are in one of the two cases, it is not very difficult to construct a category $\A$ where some $\Hom$--spaces in $\Db\A$ are proper classes, see~\cite[Exercise 1, p. 131]{F}. In fact, all $\Hom$--spaces in $\Db\A$ are sets precisely when $\Ext_\A^n(X,Y)$ are sets for each $X,Y \in \A$ and $n \ge 1$. A more detailed account of the problem and its unexpected consequences appear in \cite[\S\S2.2--2.3]{N} and \cite{CN}. 

In order to avoid the set-theoretic problems, we introduce the following definition:
\begin{defn} \label{defn:decent}
An abelian category is called \emph{decent} if for each pair of objects $X,Y \in \Db\A$ the $\Hom$--space $\Hom_{\Db\A}(X,Y)$ is a set.
\end{defn}

Next, we will recall the notions of a torsion pair and a $t$--structure. Let $\A$ be an abelian category. We say that a pair $(\T,\F)$ of full subcategories of $\A$ is called a \emph{torsion pair} in $\A$ if
\begin{enumerate}
\item $\Hom_\A(T,F) = 0$ for each $T \in \T$ and $F \in \F$;
\item For each $X \in \A$, there is a short exact sequence $0 \to T \to X \to F \to 0$ \st $T \in \T$ and $F \in \F$.
\end{enumerate}
Note that the exact sequence in $(2)$ is unique up to a unique isomorphism for each $X \in \A$. The class $\T$ is referred to as the \emph{torsion class}, while $\F$ is the \emph{torsion--free class}.

If $\D$ is a triangulated category, there is a closely related notion of a t--structure as defined in~\cite[\S 1.3]{BBD}. Let $(\Dl0,\Dr0)$ be a pair of full subcategories of $\D$. By convention, one denotes $\Dl n = \Dl0[-n]$ and $\Dr n = \Dr0[-n]$ for each $n \in \Z$. Then the pair is a \emph{t--structure} if
\begin{enumerate}
\item $\Hom_\D(X,Y) = 0$ for each $X \in \Dl0$ and $Y \in \Dr1$;
\item $\Dl0 \subseteq \Dl1$ and $\Dr0 \supseteq \Dr1$;
\item For each $Z \in \D$, there is a triangle $X \to Z \to Y \to X[1]$ \st $X \in \Dl0$ and $Y \in \Dr1$.
\end{enumerate}
Note that it follows from the axioms of a triangulated category that
the triangle in $(3)$ is unique up to a unique isomorphism. In fact, t--structures can be viewed
as a generalization of torsion pairs to the setting of  triangulated
categories, this point of view is pursued in~\cite{BR}.

Given a t--structure $(\Dl0,\Dr0)$ on $\D$, the \emph{heart} of the t--structure is defined as $\clH = \Dl0 \cap \Dr0$. The following crucial observation goes back to~\cite{BBD}:

\begin{prop} \label{prop:heart}
Let $\D$ be a triangulated category and $(\Dl0,\Dr0)$ be a t--structure with the heart $\clH = \Dl0 \cap \Dr0$. Then:
\begin{enumerate}
\item $\clH$ is an abelian category which is stable under extensions
  in $\D$ (that is, given $X,Z \in \clH$ and a triangle $X \to Y \to Z
  \to X[1]$ in $\D$, then $Y \in \clH$);
\item A sequence $0 \to X \overset{f}\to Y \overset{g}\to Z \to 0$ is
  exact in $\clH$ \iff there is a triangle $X \overset{f}\to Y
  \overset{g}\to Z \to X[1]$ in $\D$.
\item There is an isomorphism $\Ext^1_\clH(X,Y) \cong \Hom_\D(X,Y[1])$
  which is functorial in both variables.
\end{enumerate}
\end{prop}

\begin{proof}
Statement $(1)$ is included in \cite[Th\'eor\`eme
1.3.6]{BBD}. Statement $(2)$ can be easily deduced from
\cite[Proposition 1.2.2]{BBD}. Finally, $(3)$ immediately follows from
$(1)$ and $(2)$.
\end{proof}

Note that if $\A$ is an abelian category, there is a canonical t--structure on $\Db\A$ whose heart is equivalent to $\A$. 
We refer to~\cite[I.2.1]{HRS} for details.

\section{Tilting with respect to torsion pairs}
\label{sec:tilting}

In this section we will present basic facts about a tilting procedure
for abelian categories using torsion pairs. The main idea comes
from~\cite[\S\S I.2--I.4]{HRS}; an alternative approach is presented in \cite{No}. 
Our aim here is to give a streamlined
and generalized account of this topic, using the same idea
as~\cite[\S 5]{BvdB}.

We note that there have already been developed fairly general and
powerful methods for tilting and giving criteria for derived
equivalences, e.g.~\cite{R1,R2,Kel}. Our aim here is slightly
different. Many of the results either require a module category on one
side of the derived equivalence or are fairly difficult to use
for direct computations. We would like to collect and
develop enough theory here that will enable us to compute particular 
derived equivalences of general abelian categories.

We will start with recalling a crucial construction following~\cite[\S I.2]{HRS}. 
 
\begin{constr} \label{constr:tilt}
Given a decent abelian category $\A$ (see Definiton \ref{defn:decent}), and an arbitrary torsion pair $(\T,\F)$, we can construct a new abelian category $\B$ with a torsion pair $(\X,\Y)$ such that there are equivalences of full subcategories $\T \cong \Y$ and $\F \cong \X$. We proceed as follows: Let $(\Dl0,\Dr0)$ be a pair of full subcategories of $\Db\A$ defined by:
\begin{align*}
\Dl0 &= \{ \cpx X \in \Db\A \mid H^i(\cpx X) = 0 \textrm{ for all } i>0  \textrm{ and }    H^0(\cpx X) \in \T \}	\\
\Dr0 &= \{ \cpx X \in \Db\A \mid H^i(\cpx X) = 0 \textrm{ for all } i<-1 \textrm{ and } H^{-1}(\cpx X) \in \F \}	\\
\end{align*}
Here, $H^i(\cpx X)$ stands for the $i$th cohomology object of the complex $\cpx X$. That is, given
$$
\cpx X: \qquad
\dots \overset{d^{i-2}}\longrightarrow X^{i-1} \overset{d^{i-1}}\longrightarrow
X^{i} \overset{d^{i}}\longrightarrow X^{i+1} \overset{d^{i+1}}\longrightarrow \dots,
\qquad
$$
we put $H^i(\cpx X) = \Ker d^i / \Img d^{i-1}$. It is rather straightforward to check that $(\Dl0,\Dr0)$ is a t--structure, see \cite[I.2.1]{HRS} for details.

Hence $\B = \Dl0 \cap \Dr0$ is an abelian category whose exact structure is described by Proposition~\ref{prop:heart}. Note that objects of $\B$ correspond up to isomorphism to complexes $\cpx X$ in $\Db\A$ which are concentrated only in degrees $-1$ and $0$ and \st $\Ker d^{-1} \in \F$ and $\Coker d^{-1} \in \T$ (see the proof of~\cite[I.2.2]{HRS}). In fact one can view every object of $\B$, up to isomorphism, as an element of $\Ext^2_\A(T,F)$ for some $T \in \T$ and $F \in \F$. Morphisms in $\B$ are, however, more complicated, they are inherited from $\Db\A$ and correspond to equivalence classes of left fractions of homotopy chain complex maps. We refer to \cite[Chapter I]{HRS} or the beginning of~\cite[\S 4]{CF} for a more detailed description.
\end{constr}

\begin{lem} \cite[I.2.2 (b)]{HRS} \label{lem:tor_pair_tilt}
Let $\A$ be a decent abelian category, $(\T,\F)$ be a torsion pair in $\A$ and $\B \subseteq \Db\A$ be as in Construction~\ref{constr:tilt}. Then $(\F[1],\T)$ is a torsion pair in $\B$.
\end{lem}

This inspires the following definition (see \cite[p.14]{HRS}):

\begin{defn} \label{defn:tor_pair_tilt}
Let $\A$ be a decent abelian category and $(\T,\F)$ be a torsion pair
in $\A$. Let $\B$ be as in Construction~\ref{constr:tilt} and denote
$\X = \F[1]$ and $\Y = \T$ (as classes in $\Db\A$). Then we call $\B$,
resp.\ $(\B; (\X,\Y))$,
to be \emph{$(\T,\F)$--tilted} from $\A$ and put $\Phi(\A; (\T,\F)) =
(\B; (\X,\Y))$.

\end{defn}

The following three natural questions arise: First, whether $\B$ is decent, too.
 
If this is the case, then $\B$ is naturally equipped with the
torsion pair $(\X,\Y) = (\F[1],\T)$. 
The second question then asks whether we can reconstruct $\A$ from this data. 
More precisely, whether $\Phi(\B;(\X,\Y)) \cong (\A; (\T,\F))$. 

The third question asks whether the two triangulated categories which extend $\B$, 
namely $\Db\A$ and $\Db\B$, are equivalent.

We do not know a general answer to the first question, but it turns out
that even if $\B$ is decent, neither of the other two questions can be answered
affirmatively; see~\cite[I.2.3]{HRS}. The main result of this section,
Theorem \ref{thm:d_eq}, will provide a positive answer to all the three
questions above in the particular setting of tilting and cotilting classes:

\begin{defn} \label{defn:tilt_cl}
A torsion class $\T$ in an abelian category $\A$ is called \emph{tilting} if it cogenerates $\A$. That is, if for each $X \in \A$ there is a monomorphism $X \to T$ \st $T \in \T$. Dually, a torsion--free class $\F$ is called \emph{cotilting} if it generates $\A$.
\end{defn}

Note that if $\A$ has enough projectives, then a torsion--free class is cotilting \iff it contains all the projectives. A dual condition characterizes tilting torsion classes when $\A$ has enough injectives~\cite[I.3.1]{HRS}. In particular, we get:

\begin{lem} \label{lem:cotilt_cl}
Let $R$ be a right coherent ring and $(\T,\F)$ be a torsion pair in
$\modR$. Then $\F$ is a cotilting torsion--free class \iff $R \in \F$.
\end{lem}

Given an abelian category $\A$, we can also form derived categories for subcategories of $\A$ relative to $\A$. We make this precise in the following definition:

\begin{defn} \label{defn:der_rel}
Let $\E$ be a full subcategory of a decent abelian category $\A$ \st $\E$ is closed under finite coproducts in $\A$. Then we denote $\Kbac{\E;\A} = \Kb\E \cap \Kbac\A$. The \emph{derived category of $\E$ relative to $\A$} is defined as the Verdier quotient $\Db{\E;\A} = \Kb\E / \Kbac{\E;\A}$.
\end{defn}

In other words, we add formal inverses to all morphisms in the homotopy category of complexes $\Kb\E$ which are quasi--isomorphisms in $\Kb\A$. Note that again it is not clear in general whether the $\Hom$--spaces in $\Db{\E;\A}$ are sets. We will, however, show that they are in the situation we are interested in. Ignoring this for the moment and using the universal localization property, we see that the full embedding $\Kb\E \to \Kb\A$ uniquely extends to a functor $\Db{\E;\A} \to \Db\A$. We will give a criterion for this functor to be a triangle equivalence, but we need one more definition first.

\begin{defn} \label{defn:fin_res}
Let $\E$ be a full subcategory of an abelian category $\A$. We say that $\A$ \emph{admits finite $\E$--resolutions} if for each $X \in \A$ there is a finite exact sequence
$$ 0 \to E_n \to \dots \to E_1 \to E_0 \to X \to 0 $$
\st $E_i \in \E$ for each $0 \leq i \leq n$. Similarly, we say that $\A$ \emph{admits finite $\E$--coresolutions} if for each $X \in \A$ there is a finite exact sequence
$$ 0 \to X \to E'_0 \to E'_1 \to \dots \to E'_m \to 0 $$
\st $E'_j \in \E$ for each $0 \leq j \leq m$.
\end{defn}

\begin{lem} \label{lem:fin_res}
Let $\A$ be a decent abelian category and $\E$ be a full subcategory
closed under finite coproducts. Suppose that $\A$ admits finite
$\E$--resolutions or coresolutions. Then the full embedding $\Kb\E \to
\Kb\A$ induces a triangle equivalence $F: \Db{\E;\A} \to \Db\A$. In
particular, all $\Hom$--spaces in $\Db{\E;\A}$ are sets.
\end{lem}

\begin{proof}
We will prove only the case when $\A$ admits finite $\E$--coresolutions, the other case being dual. Since the result is central for the considerations below, we prefer to give a detailed proof. 

Following~\cite[I.4.6]{Har}, we will first show that for any complex $\cpx X \in \Kb\A$ there is a quasi--isomorphism $t: \cpx X \to \cpx Y$ in $\Kb\A$ \st $\cpx Y \in \Kb\E$. If $\cpx X$ is a complex of objects of $\A$ \st $X^i = 0$ for $i<p$ and $i>q$, we first construct morphisms $t^i: X^i \to Y^i$ by induction for $i < q$. Let $Y^i = 0$ and $t^i = 0$ for $i<p$, and $t^p: X^p \to Y^p$ be a monomorphism into some $Y^p \in \E$. Such a monomorphism must exist by assumption. Given $t^i$, we construct $t^{i+1}$ by composing the morphism in the second column of the pushout diagram
$$
\begin{CD}
\Coker d^{i-1}_X	@>>>	X^{i+1}		\\
@V{\bar t^i}VV			@VVV		\\
\Coker d^{i-1}_Y	@>>>	P^{i+1}		\\
\end{CD}
$$
with a monomorphism $P^{i+1} \to Y^{i+1}$ \st $Y^{i+1} \in \E$. Here, $d^{i-1}_Y: Y^{i-1} \to Y^i$ is the obvious morphism coming from the preceding step of the construction, and $\bar t^i$ is the cokernel morphism constructed using the diagram
$$
\begin{CD}
X^{i-1} 	@>{d_X^{i-1}}>>	X^i 	@>>>	\Coker d^{i-1}_X	@>>>	0	\\
@V{t^{i-1}}VV			@V{t^i}VV	@V{\bar t^i}VV				\\
Y^{i-1} 	@>>{d_Y^{i-1}}>	Y^i 	@>>>	\Coker d^{i-1}_Y	@>>>	0	\\
 \end{CD}
$$
Finally, we complete the complex $\cpx Y$ by a finite $\E$--coresolution $0 \to P^q \to Y^{q+1} \to Y^{q+2} \to \dots \to Y^s \to 0$, and put $Y^j = 0$ for all $j>s$. Note that all the components of $t: \cpx X \to \cpx Y$ are monomorphisms in $\A$ and $t$ is easily seen to be a quasi--isomorphism. This in particular shows that the functor $F$ is dense.

Next we use the same argument as in the proof of~\cite[I.3.3]{HRS} to show that $F$ is full. Let $\sigma\inv f: \cpx X \to \cpx Y$ be a morphism in $\Db\A$ \st $\cpx X, \cpx Y \in \Kb\E$, 
$f: \cpx X \to \cpx Z$ is a morphism in $\Kb\A$, and $\sigma : \cpx Y \to \cpx Z$ is a quasi--isomorphism. 
Since $F$ is dense, there exists a quasi--isomorphism
$t: \cpx Z \to \cpx W$  \st $\cpx W \in \Kb\E$. Then $F( (t\sigma)\inv (tf) ) = (t\sigma)\inv (tf) = \sigma\inv f$.

To prove that $F$ is faithful, assume that $F(\sigma\inv f) = 0$ where $\cpx X \overset{f}\to \cpx Z \overset{\sigma}\leftarrow \cpx Y$ is a morphism in $\Db{\E;\A}$. This is precisely to say that there is a quasi--isomorphism $s: \cpx Z \to \cpx V$ in $\Kb\A$ \st $sf = 0$. Again, there is a quasi--isomorphism $t: \cpx V \to \cpx W$ with $\cpx W \in \Kb\E$. Consequently $tsf = 0$ and $ts: \cpx Z \to \cpx W$ is a quasi--isomorphism which is contained in $\Kb\E$. This precisely says that $\sigma\inv f = 0$. Hence $F$ is a triangle equivalence.

Finally, since $\A$ is decent and we have constructed the isomorphisms
$$ \Hom_{\Db{\E;\A}}(X,Y) \to \Hom_{\Db\A}(X,Y) $$
for each pair of objects $X,Y \in \Db{\E;\A}$, the $\Hom$--spaces $\Hom_{\Db{\E;\A}}(X,Y)$ must be sets.
\end{proof}

\begin{rem}\label{set-theretic} 
In the proof above, the functor $F$ was shown to be fully faithful and dense. In order to construct a quasi--inverse of $F$ in case $\mathcal A$ is not skeletally small, one needs the Axiom of Choice for proper classes, hence has to work in the von Neumann--Bernays--G\"odel axiomatic set theory rather than ZFC. For more details, we refer to \cite[\S1]{St}.
\end{rem}

If $\E$ is closed under extensions in $\A$, it is, together with the exact sequences inherited from $\A$, an \emph{exact category} -- a concept originally defined by Quillen and well described in~\cite[Appendix A]{Kel2}. In this case, one can define the bounded derived category of $\E$ in the sense of~\cite{N2}.
The following easy lemma shows that if $\E$ is torsion or a torsion free class, then this derived category coincides with $\Db{\E;\A}$ and, in particular, to construct $\Db{\E;\A}$ one only needs to be able to identify short exact sequences in $\E$.

\begin{lem} \label{lem:acycl}
Let $\A$ be an abelian category and $\E$ be either torsion or a torsion--free class. Consider a complex
$$
\cpx X: \qquad
\dots \overset{d^{i-2}}\longrightarrow X^{i-1} \overset{d^{i-1}}\longrightarrow
X^{i} \overset{d^{i}}\longrightarrow X^{i+1} \overset{d^{i+1}}\longrightarrow \dots,
\qquad
$$
in $\Kbac{\E;\A}$. Then $\Coker d^i \in \E$ for each $i \in \Z$.
\end{lem}

\begin{proof}
This is obvious.
\end{proof}

Before stating the main result of this section, we will need an important statement, originally from~\cite{HRS}:

\begin{prop} \cite[I.3.2]{HRS} \label{prop:tor_pair_tilt2}
Let $\A$ be a decent abelian category, $(\T,\F)$ be a torsion pair in $\A$, and $\Phi(\A; (\T,\F)) = (\B; (\X,\Y))$ as in Definition~\ref{defn:tor_pair_tilt}.
\begin{enumerate}
\item If $\T$ is a tilting torsion class, then $\Y$ is a cotilting torsion--free class.
\item If $\F$ is a cotilting torsion--free class, then $\X$ is a tilting torsion class.
\end{enumerate}
\end{prop}

\begin{proof}
Although this result has been shown in~\cite{HRS} or~\cite[\S 4]{CF}, we prefer to give a simple direct proof here. Thus we also avoid a minor omission at the beginning of page~18 in~\cite{HRS} -- one needs an extra argument for making $\pi$ to an epimorphism in $\B$ there. We will prove only $(1)$, the statement of $(2)$ is dual.

Assume that $\T$ cogenerates $\A$ and recall that $\T = \Y$ by definition. Let $\cpx X \in \B$; we can without loss of generality assume that $X^i = 0$ for all indices $i$ except for $i=-1$ and $0$, as dicussed before. In this case, $\cpx X$ is completely given by a morphism $d^{-1}: X^{-1} \to X^0$. Let us denote $F = \Ker d^{-1}$ and $T = \Coker d^{-1}$; then $T \in \T$ and $F \in \F$ by assumption. We can further assume that both $X^{-1}$ and $X^0$ are in $\T$. If they are not, we pass to a quasi-isomorphic complex $\cpx{\tilde X}$ by taking an embedding $f: X^{-1} \to \tilde X^{-1}$ in $\A$ with $\tilde X^{-1} \in \T$ and forming the following push-out in $\A$:
$$
\begin{CD}
X^{-1}		@>{d^{-1}}>>		X^0		\\
@V{f}VV					@VVV		\\
\tilde X^{-1}	@>{\tilde d^{-1}}>>	\tilde X^0	\\
\end{CD}
$$
Notice that automatically $\tilde X^0 \in \T$ since clearly $\Img \tilde d^{-1}, \Coker \tilde d^{-1} \in \T$. The argument just presented is in fact a short account on~\cite[Lemma 4.4]{CF}. 
But now, if $X^{-1}, X^0 \in \T = \Y$, then the obvious triangle in $\Db\A$:
$$
\begin{CD}
X^{-1} @>{d^{-1}}>> X^0 @>>> \cpx X @>>> X^{-1}[1]
\end{CD}
$$
induces, using Proposition~\ref{prop:heart}, the short exact sequence
$$
\begin{CD}
0 @>>> X^{-1} @>{d^{-1}}>> X^0 @>>> \cpx X @>>> 0
\end{CD}
$$
in $\B$. Hence every $\cpx X \in \B$ is an epimorphic image of an object from $\Y$ in the category $\B$, and $\Y$ is a cotilting torsion--free class in $\B$.
\end{proof}

Now we are in a position to state the main result of this section which gives a positive answer to the three questions above in the tilting and cotilting cases. It is a generalization of~\cite[I.3.3 and I.3.4]{HRS} which have some extra assumptions regarding existence of projective or injective objects. These assumptions turn out to be unnecessary which makes applications of the theorem considerably easier. In fact, the same idea as we are going to present below was used in~\cite[\S 5]{BvdB} for equivalence of unbounded derived categories. For the convenience of the reader, we provide here more details for the bounded case. 

We will use the notation $(\A; (\T,\F)) \cong (\A'; (\T',\F'))$ for two abelian categories $\A,\A'$ with the respective torsion pairs such that there exists an equivalence $F: \A \to \A'$ which by restriction induces equivalences $\T \to \T'$ and $\F \to \F'$.

\begin{thm} \label{thm:d_eq}
Let $\A$ be a decent abelian category and $(\T,\F)$ a torsion pair in $\A$. Let $\B$ be the $(\T,\F)$--tilted abelian category as in Definition~\ref{defn:tor_pair_tilt}, and let $(\X,\Y) = (\F[1],\T)$. If either $\T$ is tilting or $\F$ is cotilting, then:
\begin{enumerate}
\item There is a triangle equivalence functor $F: \Db\A \to \Db\B$ extending the identity functor on $\B$;
that is, $F \restriction \B = id_\B$.
\item $\Phi(\B; (\X,\Y)) \cong (\A; (\T,\F))$; that is, $\A$ is $(\X,\Y)$--tilted from $\B$.
\end{enumerate}
\end{thm}

\begin{proof}
We will only give a proof for the case when $\T$ is a tilting torsion class in $\A$. The other case is dual.

$(1)$. If $\T$ is tilting, then there is an exact sequence $0 \to X \to T_0 \to T_1 \to 0$ \st $T_0, T_1 \in \T$ for each $X \in \A$. In particular, $\A$ admits finite $\T$--coresolutions. Similarly, $\B$ admits finite $\T$--resolutions since $\T \subseteq \B$ is a cotilting torsion--free class by Proposition~\ref{prop:tor_pair_tilt2}. Note also that a sequence $0 \to X \to Y \to Z \to 0$ in $\T$ is exact in $\A$ \iff it is exact in $\B$ by Proposition~\ref{prop:heart}. Hence $\Kbac{\T;\A} = \Kbac{\T;\B}$ by Lemma~\ref{lem:acycl}. Consequently, we obtain triangle equivalences
$$ \Db\A \overset{\sim}\longleftarrow \Db{\T;\A} = \Db{\T;\B} \overset{\sim}\longrightarrow \Db\B $$
by Lemma~\ref{lem:fin_res}. This yields a triangle equivalence $F: \Db\A \to \Db\B$ which, without loss of generality, acts as the identity functor on the full subcategory given by complexes with components in $\T$. But, as shown in~\cite[4.4]{CF} and recalled in the proof of Proposition~\ref{prop:tor_pair_tilt2}, each $\cpx X \in \B$ is isomorphic to such a complex. Hence we may take $F$ \st $F \restriction \B = id_\B$.

$(2)$. A detailed proof for $\Phi(\B; (\F[1],\T)) \cong (\A; (\T,\F))$ is given in~\cite[I.3.4]{HRS}. We just have to substitute the use of~\cite[I.3.3]{HRS} in that proof by the first part of this theorem. In fact, if $\A'$ is $(\F[1],\T)$--tilted from $\B$ and $G: \Db\B \to \Db\A$ is a quasi--inverse of $F$, it is shown in~\cite[p.20]{HRS} that the restriction of $G$ to $\A'$ induces an equivalence $\A' \to \A[1]$ which respects the corresponding torsion pairs.
\end{proof}

\section{Tilting objects}
\label{sec:tilt_obj}

Having defined and described the tilting process via torsion pairs, we
shall consider the case when the tilted category is a module
category. This leads to the concept of a tilting object. We will
consider only skeletally small abelian categories in this context, although there is
an analogue for non--small abelian categories, too. We will shortly
discuss this at the end of the section.

\begin{defn} \label{defn:tilt_obj}
Let $\A$ be a skeletally small abelian category. Then $T$ is a
\emph{tilting object in $\A$} if there is a tilting torsion class
$\T \subseteq \A$ \st $T$ becomes a projective generator in the
$(\T,\F)$--tilted abelian category $\B$. That is:
\begin{enumerate}
\item $T$ is contained in $\B$ and is projective there,
\item $\B = \gen T$, where $\gen T$ stands for the full subcategory
  formed by all epimorphic images of finite coproducts of copies of $T$.
\end{enumerate}
\end{defn}

Note that $T \in \T$ by Lemma~\ref{lem:cotilt_cl} since $\T$ is a cotilting torsion--free class in $\B$. Moreover, the functor $\Hom_\B(T,-): \B \to \modS$, where $S = \End_\A(T) = \End_\B(T)$, is a category equivalence, \cite[II.2.5]{ARS}. As a consequence, $S$ must be right coherent and we get the triangle equivalence
$$
F: \Db\A \overset{\sim}\longrightarrow \Db\modS.
$$
In fact, one can show that $F \cong \RHom_\A(T,-)$; we refer
to~\cite[\S 3]{Kel3} for introduction to derived functors.

In view of Theorem~\ref{thm:d_eq} we have, for a given right coherent
ring $S$, a description (up to equivalence) of all small abelian
categories $\A$ with a tilting object $T$ \st $\End_\A(T) \cong S$.
Namely, every such category is tilted from $\modS$ by a
torsion pair $(\T,\F)$ in $\modS$ with $S \in \F$. Then $T = S[1]$ is
the corresponding tilting object in $\A$. To illustrate this, we
classify all small abelian categories $\A$ with an (indecomposable) tilting
object $T$ \st $\End_\A(T) \cong \Z$.

\begin{expl} \label{expl:modZ}
Let $S = \Z$. Then the cotilting torsion--free classes in $\modZ$ are
parametrized by subsets of the set $\mathbb P$ of all prime
numbers. More precisely, if $Q \subseteq \mathbb{P}$, we take the
torsion pair $(\X_Q,\Y_Q)$ \st $\X_Q$ is the class of all finite
abelian groups whose orders have prime factors only in $Q$.

Let us denote by $\A_Q$ the $(\X_Q,\Y_Q)$--tilted category from
$\modZ$. In this way we obtain a continuum of abelian categories. It
is easy to see that they are mutually non--equivalent and it will be
shown in the next section that they are all hereditary. Moreover,
one can easily describe isomorphism classes of all objects of each
$\A_Q$ and morphisms between them.
\end{expl}

As opposed to the purpose--oriented Definition~\ref{defn:tilt_obj},
one can also determine tilting objects in an abelian category $\A$
directly. The conditions given in the following proposition extend the
definition used by Happel in~\cite{H4} for locally finite
hereditary abelian categories.

\begin{prop} \label{prop:tilt_obj}
Let $\A$ be a small abelian category and $T \in \A$. Then $T$ is a
tilting object \iff
\begin{enumerate}
\item $\pd_\A T \leq 1$,
\item $\Ext^1_\A(T,T) = 0$,
\item $\Hom_\A(T,X) = 0 = \Ext^1_\A(T,X)$ implies $X = 0$.
\item $\Hom_\A(T,X)$ and $\Ext^1_\A(T,X)$ are finitely generated as $\End_\A(T)$--modules.
\end{enumerate}
In $(3)$ and $(4)$, $X$ runs over all objects in $\A$.
\end{prop}

\begin{proof}
Most of the arguments here have been used by several authors before, but we
recall the whole proof for the reader's convenience. Condition $(4)$
is clearly equivalent to the fact that for each $X \in \A$ there
exist:
\begin{enumerate}
\item[(a)] a morphism $p_X: T^{n_X} \to X$ \st $\Hom_\A(T,p_X)$ is surjective;
\item[(b)] an exact sequence $0 \to X \to E_X \to T^{m_X} \to 0$ \st the
  connecting homomorphism $\Hom_\A(T,T^{m_X}) \to \Ext^1_\A(T,X)$ is surjective.
\end{enumerate}

Assume that $T \in \A$ satisfies conditions $(1)$ -- $(4)$.
Using $(1)$ and $(2)$, one easily checks that all $\Ext^1_\A(T,\Img p_X)$,
$\Ext^1_\A(T,\Ker p_X)$ and $\Ext^1_\A(T,E_X)$ vanish. Let us define full
subcategories of $\A$ as follows:
\begin{align*}
\T &= \{ U \mid p_U: T^{n_U} \to U \textrm{ is an epimorphism in } \A \}, \\
\F &= \{ F \mid \Hom_\A(T,F) = 0 \}.
\end{align*}
We claim that $(\T,\F)$ is a tilting torsion pair in $\A$ and $\T = \{
U \mid \Ext^1_\A(T,U) = 0\}$. Clearly $\Hom_\A(U,F) = 0$ for each $U \in
\T$ and $F \in \F$. Moreover, $\T$ can easily be shown to be closed
under extensions using the same idea as for the horseshoe lemma. It
follows that for each $X \in \A$, there is a short exact sequence $0
\to tX \to X \to fX \to 0$ \st $tX = \Img p_X \in \T$ and $fX \in
\F$. Hence $(\T,\F)$ is a torsion pair in $\A$. Clearly, $\Ext^1_\A(T,U)
= 0$ for each $U \in \T$. On the other hand, if $\Ext^1_\A(T,X) = 0$, then
$\Hom_\A(T,fX) = 0 = \Ext^1_\A(T,fX)$, so $fX = 0$ and $X \in \T$ by
$(3)$. Finally, (b) shows that each $X \in \A$ embeds into some
$E_X \in \T$. This proves the claim.

Let $\B$ be $(\T,\F)$--tilted from
$\A$. Theorem~\ref{thm:d_eq} yields isomorphisms
\begin{align*}
\Ext^1_\B(T,F[1]) &\cong \Hom_{\Db\A}(T,F[2]) = \Ext^2_\A(T,F) = 0 \\
\Ext^1_\B(T,U)    &\cong \Hom_{\Db\A}(T,U[1]) = \Ext^1_\A(T,U) = 0
\end{align*}
for each $U \in \T$ and $F \in F$. Hence $T$ is projective in $\B$
since $(\F[1],\T)$ is a torsion pair in $\B$ by
Lemma~\ref{lem:tor_pair_tilt}. It remains to prove that $T$ generates
$\B$. We know that $\T$ generates $\B$ by
Proposition~\ref{prop:tor_pair_tilt2}. Moreover, for any $U \in \T$
the short exact sequence
$$
\begin{CD}
0 @>>> \Ker p_U @>>> T^{n_U} @>{p_U}>> U @>>> 0 
\end{CD}
$$
in $\A$ has all terms in $\T$, so it is also a short exact sequence in
$\B$. Hence $T$ generates $\B$ and, consequently, $T$ is a tilting object
in $\A$ in the sense of Definition~\ref{defn:tilt_obj}.

The converse statement that every tilting object $T \in \A$ satisfies conditions
$(1)$ to $(4)$ is straightforward. One uses Theorem~\ref{thm:d_eq} and
the triangle equivalences
$$ \Db\A \overset{\sim}\longrightarrow \Db\B \overset{\sim}\longrightarrow \Db\modS, $$
where $S = \End_\A(T)$. We just note that the $S$--modules $\Hom_A(T,X)$ and
$\Ext^1_\A(T,X)$ are realized as homologies in degrees $0$ and $1$, respectively,
of the image of $X$ under the equivalence
$F: \Db\A \overset{\sim}\longrightarrow \Db\modS$. This is because
$F(T) = S$ and $\Hom_{\Db\A}(T,X[i]) \cong \Hom_{\Db\modS}(S,FX[i]) \cong H^i(FX)$.
\end{proof}

There are two main sources of examples of tilting objects according to
our definition, which appear in the literature:
\begin{enumerate}
\item If $T$ is a tilting object in a locally finite abelian category
  $\A$ in the sense of Happel, Reiten and Smal\o{}~\cite[I.4]{HRS},
  then $T$ is also a tilting object according to Definition~\ref{defn:tilt_obj}.

\item If $T$ is a $1$--tilting $R$--module in the sense of
  Miyashita~\cite{Mi} and $R$ is right
  noetherian, then $T$ is a tilting object in $\modR$ in the
  sense of Definition~\ref{defn:tilt_obj} (see \cite[Theorem 2.5(ii)]{Tr}). 
\end{enumerate}

\begin{rem}\label{coherent} In fact, $1$--tilting $R$--modules in the sense of
Miyashita are tilting objects in the sense of \ref{defn:tilt_obj} even in the 
more general case when both $R$ and $S = \End_R(T)$ are right coherent.
This can be proved using \cite[Proposition 8.1]{R1}.
\end{rem}

\medskip
 
Now, we shall briefly discuss decomposition properties of objects in abelian
categories with a tilting object. 

\begin{lem} \label{lem:db_decomp}
Let $R$ be a right noetherian ring. Then
\begin{enumerate}
\item Any chain $\cpx X_0 \to \cpx X_1 \to \cpx X_2 \to \dots$ of
  split epimorphisms in $\Db\modR$ stabilizes.
\item Any $\cpx X \in \Db\modR$ decomposes to a finite coproduct of
  indecomposable objects.
\end{enumerate}
\end{lem}

\begin{proof}
It is rather well known that $(1)$ implies $(2)$. If we have a chain
as in $(1)$, we get chains of split epimorphisms of homologies:
$$
H^i(\cpx X_0) \longrightarrow
H^i(\cpx X_1) \longrightarrow
H^i(\cpx X_2) \longrightarrow \dots
$$
All but finitely many of those chains consist only of zero objects and
each of those finitely many non--zero chains stabilizes since $R$ is
right noetherian. Hence, there is some $N > 0$ \st
$H^i(\cpx X_j) \to H^i(\cpx X_{j+1})$ is an isomorphism for each
$i \in \Z$ and $j > N$. Consequently, $\cpx X_j \to \cpx X_{j+1}$ is
an isomorphism in $\Db\modR$ for each $j > N$ since it induces
isomorphisms on all homologies.
\end{proof}

If the endomorphism ring of the tilting object 
is right noetherian, we have the following:

\begin{prop} \label{prop:tilt_decomp}
Let $\A$ be a small abelian category with a tilting object $T$ \st
$\End_\A(T)$ is right noetherian. Then
\begin{enumerate}
\item Every chain of split epimorphisms in $\A$ stabilizes.
\item Every object $X \in \A$ decomposes into a finite coproduct of
  indecomposables.
\end{enumerate}
\end{prop}

\begin{proof}
This follows immediately from Lemma \ref{lem:db_decomp}, using the triangle equivalence $\Db\A \to
\Db\modS$ where $S = \End_\A(T)$.
\end{proof}

Note that the decomposition given by Lemma~\ref{lem:db_decomp}
or Proposition~\ref{prop:tilt_decomp} is in general not unique in the
sense of Krull--Schmidt. Moreover,
neither of the two statements hold true for general coherent
rings. To see this, let $R$ be any von Neumann regular (hence
coherent) ring which is not artinian. Then there is always a
strictly descending chain of split epimorphisms of the form
$$
R \longrightarrow
e_1 R \longrightarrow
e_2 R \longrightarrow
e_3 R \longrightarrow \dots
$$
However, there are cases when every chain of split epimorphisms
stabilizes even for non--noetherian objects, as we show in the
following examples:

\begin{expl} \label{expl:non-noe1}
Let $R$ be a right noetherian ring with an $1$--tilting module
$T \in \modR$ whose endomorphism ring $S$ is right coherent, but
\emph{not} right noetherian. Then $\Db\modR$ and $\Db\modS$ are
equivalent, so $\modS$ has the chain condition on split epimorphisms 
by Proposition \ref{prop:tilt_decomp}.
Examples of this kind were constructed by Valenta \cite{Va} 
using \cite[Example to 3.7]{Ta} (see also \cite[Example 2.8]{Tr}).
\end{expl}

\begin{expl} \label{expl:non-noe2}
Let $Q$ be a non-empty set of prime numbers and $\A_Q$ be the abelian
category from Example~\ref{expl:modZ}. Let $T = \Z[1] \in \A_Q$ be the
tilting object. Then for any $p \in Q$ the triangle
$\Z \overset{p}\to \Z \to \Z/p\Z \to \Z[1]$ in $\Db\modZ$ induces,
via Theorem~\ref{thm:d_eq} and Proposition~\ref{prop:heart}, the short
exact sequence
$$
\begin{CD}
0 @>>> \Z/p\Z @>>> T @>>> T @>>> 0
\end{CD}
$$
in $\A_Q$. Hence $T$ definitely is not a noetherian object in $\A_Q$,
but $\A_Q$ still has the chain condition on split epimorphisms.
\end{expl}

We conclude the section with a short remark on tilting objects for
non--small abelian categories.

\begin{rem}
We can adjust Definition~\ref{defn:tilt_obj} for the case when $\A$ is
a decent AB4 abelian category. We can call $T \in \A$ a
tilting object if $T$ becomes a self--small projective generator in
some $(\T,\F)$--tilted category $\B$. Then
necessarily $\B$ is equivalent to $\ModS$ for $S = \End_\A(T)$. It
can be shown that this definition is equivalent to Colpi's and
Fuller's definition from~\cite{CF}. In particular, any abelian
category with a tilting object in the sense of~\cite{CF} is AB4 and
decent; see also~\cite[3.2]{CGM}.
\end{rem}

\section{Tilting from hereditary categories}
\label{sec:tilt_hered}

As mentioned in the Introduction, the main result of \cite{HRS} characterizes all artin algebras whose module categories can be tilted from (or to) a locally finite hereditary abelian category. We aim to extend this characterization to all right noetherian rings. However, in this section, we actually pursue a more general goal of characterizing all decent abelian cateogries which can be tilted to a hereditary abelian category.

Recall that a torsion pair $(\X,\Y)$ in an abelian category $\B$ is \emph{split} if $\Ext^1_\B(\Y,\X) = 0$. That is, for each $Z \in \B$ the exact sequence $0 \to X \to Z \to Y \to 0$ with $X \in \X$ and $Y \in \Y$ splits. We start with an easy lemma.

\begin{lem} \label{lem:q_hered_catg}
Let $\B$ be an abelian category and $(\X,\Y)$ be a torsion pair in $\B$ \st $\Y$ is cotilting and $\pd_\B Y \leq 1$ for each $Y \in \Y$. Then $\gldim \B \leq 2$.
\end{lem}

\begin{proof}
Let $Z \in \B$. Since $\Y$ is a cotilting torsion--free class, there is a short exact sequence $0 \to Y_1 \to Y_0 \to Z \to 0$ \st $Y_0, Y_1 \in \Y$. Applying $\Hom_\B(-,W)$ for any $W \in \B$, we get an exact sequence:
$$ 0 = \Ext^2_\B(Y_1,W) \to \Ext^3_\B(Z,W) \to \Ext^3_\B(Y_0,W) = 0 $$
Hence $\Ext^3_\B(-,-) \equiv 0$ and $\gldim\B \leq 2$.
\end{proof}

Now we can state and prove the main result of the section. Note that there is also a dual version with the tilting torsion class replaced by a cotilting torsion--free class and projective dimension by injective dimension.

\begin{thm} \label{thm:tilt_hered}
Let $\A$ be a decent abelian category and $(\T,\F)$ a torsion pair in $\A$ \st $\T$ is a tilting torsion class. Let $\B$ be the $(\T,\F)$--tilted abelian category and denote $(\X,\Y) = (\F[1],\T)$. Then the following are equivalent:
\begin{enumerate}
\item $\A$ is a hereditary abelian category (that is, $\Ext^2_\A(-,-) \equiv 0$).
\item $(\X,\Y)$ is a split torsion pair in $\B$ and $\pd_\B Y \leq 1$ for each $Y \in \Y$.
\end{enumerate}
\end{thm}

\begin{proof}
Let $\Phi(\A; (\T,\F)) = (\B; (\X,\Y))$ be as in the premise. Consider
$T_1, T_2 \in \T = \Y$ and $F_1, F_2 \in \F$. Note that then $F_1[1],
F_2[1] \in \F[1] = \X$. Using Proposition~\ref{prop:heart}
and Theorem~\ref{thm:d_eq} we deduce the following formulas:
\begin{align}
\Ext^n_\A(T_1,T_2) &\cong \Hom_{\Db\B}(T_1,T_2[n]) = \Ext^n_\B(T_1,T_2)			\\
\Ext^n_\A(T_1,F_2) &\cong \Hom_{\Db\B}(T_1,F_2[1][n-1]) = \Ext^{n-1}_\B(T_1,F_2[1])	\\
\Ext^n_\A(F_1,T_2) &\cong \Hom_{\Db\B}(F_1[1],T_2[n+1]) = \Ext^{n+1}_\B(F_1[1],T_2)
\end{align}

If $\A$ is hereditary then $(ii)$ used for $n=2$ implies that $(\X,\Y) = (\F[1],\T)$ is a split torsion pair in $\B$. It follows immediately from $(i)$ and $(ii)$ that $\Ext^2_\B(T_1,\Y) = 0$ and $\Ext^2_\B(T_1,\X) = 0$. Since $(\X,\Y)$ is a torsion pair in $\B$, we get $\pd_\B T_1 \leq 1$. This finishes the proof of $(1) \implies (2)$.

Conversely, assume $(2)$. As $(\T,\F)$ is a torsion pair in $\A$, we only must prove that $\Ext^2_\A(Z,W) = 0$ whenever $Z$ is either in $\T$ or in $\F$ and $W$ is either in $\T$ or in $\F$. We are, therefore, left with four cases. First note that $\gldim\B \leq 2$ by Proposition \ref{prop:tor_pair_tilt2} and Lemma~\ref{lem:q_hered_catg}, so $\Ext^2_\A(\F,\T) = 0$ by $(iii)$. $\Ext^2_\A(\T,\T) = 0$ and $\Ext^2_\A(\T,\F) = 0$ follow immediately by the assumption on $\B$ using $(i)$ and $(ii)$, respectively. Finally, consider $F_1,F_2 \in \F$. Since $\T$ is a tilting torsion class in $\A$, there is an exact sequence $0 \to F_1 \to T \to T' \to 0$ in $\A$ \st $T,T' \in \T$. Applying $\Hom_\A(-,F_2)$ we obtain the exact sequence:
$$ \Ext^2_\A(T,F_2) \to \Ext^2_\A(F_1,F_2) \to \Ext^3_\A(T',F_2). $$

Now, we have already proved that the first term vanishes, and $\Ext^3_\A(T',F_2) \cong \Ext^2_\B(T',F_2[1]) = 0$ by $(ii)$ and by the assumption of $\pd_\B T' \leq 1$. Consequently $\Ext^2_\A(F_1,F_2) = 0$. Hence $\A$ is hereditary and $(2) \implies (1)$ is proved.
\end{proof}

In view of our results in Section 3, the latter theorem might be restated as follows: Assume that we start with an abelian category $\B$ with a cotilting torsion--free class $\Y$. Then we get a hereditary category by $(\X,\Y)$--tilting $\B$ precisely when $(\X,\Y)$ is a split torsion pair in $\B$ and all objects in $\Y$ have projective dimension at most $1$. Hence we have the following corollary, proving the equivalence (i)~$\Longleftrightarrow$~(ii) of the Main Theorem for all coherent rings:

\begin{cor} \label{cor:rings}
The following are equivalent for a right coherent ring $R$:
\begin{enumerate}
\item $R$ is isomorphic to the endomorphism ring of a tilting object
  in a skeletally small hereditary abelian category.
\item There is a split torsion pair in $\modR$ whose torsion--free class $\mathcal Y$ consists of modules of projective dimension $\leq 1$, and $R \in \mathcal Y$.
\end{enumerate}
\end{cor}

\begin{expl} \label{expl:modZ_cont}
All the categories $\A_Q$, $Q \subseteq \mathbb{P}$, from Example~\ref{expl:modZ} are hereditary because all the torsion pairs $(\X_Q,\Y_Q)$ in $\modZ$ are split. The condition of $\pd_\Z Y \leq 1$ for each $Y \in \Y_Q$ is trivially satisfied since $\Z$ is a hereditary ring.
\end{expl}

\section{Almost hereditary rings}
\label{sec:q_tilt}

The aim of this section is to prove the equivalence of conditions (ii) and (iii) of the Main Theorem. We start with the easier implication\footnote{After proving the equivalence of (ii) and (iii) in the Main Theorem, we learned that the first part of our proof, Lemmas \ref{lem:back} and \ref{largest}, had independently and earlier been obtained by Colpi, Fuller, and Gregorio. So we credit these two results to them.}.
Recall that we call a right noetherian ring $R$ \emph{almost hereditary} if $R$ has right global dimension $\leq 2$, and $\pd M \leq 1$ or $\id M \leq 1$ for each finitely generated indecomposable module $M$.

\begin{lem}[Colpi, Fuller, Gregorio]\label{lem:back} 
Let $R$ be a right noetherian ring with a split torsion pair $(\mathcal X, \mathcal Y)$ in $\modR$ such that $R \in \mathcal Y$ and $\mathcal Y$ consists of modules of projective dimension $\leq 1$. Then $R$ is almost hereditary.
\end{lem}

\begin{proof}
Our assumptions imply $\gldim(\modR) \leq 2$ by Lemma~\ref{lem:q_hered_catg}. By Auslander Lemma, the right global dimension of $R$ is the supremum of projective dimensions of cyclic right $R$--modules, so $\rgldim R = \gldim(\modR)$.

Since $R$ is right noetherian, each module $M \in \modR$ is a finite (but not necessarily unique) direct sum of modules from $\indR$.

Now assume there is $M \in \indR$ such that $\pd_R M = \id_R M = 2$. Then $M \in \X$. Since $\id_R M = 2$, Baer's Criterion gives a right ideal $I$ of $R$ such that $\Ext^1_R(I,M) \neq 0$.

This is a contradiction to $(\X,\Y)$ being split since $I$ must, as a submodule of $R$, belong to $\Y$.
\end{proof}

\bigskip
Now we start with the proof of the implication (iii) $\implies$ (ii) of the Main Theorem. This is trivial when $\rgldim R \leq 1$ (just take $\mathcal X = \{ 0 \}$ and $\mathcal Y = \modR$). So \emph{for the rest of this section, we will assume that $R$ is a right noetherian ring with $\rgldim R = 2$}.

\bigskip
In particular, if $\mathcal P_1$ denotes the class of all modules of projective dimension $\leq 1$, then $\mathcal P_1$ will be closed under submodules.

By induction on $n$, we define the classes of indecomposable modules $\C_n$ as follows: $\C_0$ is the class of all $M \in \indR$ with $\pd_R M = 2$, and $\C_{n+1}$ the class of all modules $M \in \indR$ such that 
$\Hom_R(P,M) \neq 0$ for some $P \in \C_n$. Let $\C = \bigcup_{n < \omega} \C_n$. 
Notice that this construction has the property that for each $M \in \indR$ we have $M \in \C$ if and only if $\Hom_R(\C,M) \neq 0$.

Let $\Y_0 = \{ M \in \modR \mid \Hom_R(\C,M) = 0 \}$ and $\X _0 = \{ N \in \modR \mid \Hom_R(N,\Y_0) = 0 \}$.
Then $(\X _0, \Y_0)$ is a torsion pair with $\C \subseteq \X_0$.
\medskip

\begin{lem}[Colpi, Fuller, Gregorio] \label{largest} 
$(\X_0, \Y_0)$ is a split torsion pair in $\modR$ such that $\X_0 = \add\C$ and $\pd_R Y \leq 1$ for each $Y \in \Y_0$.
\end{lem}

\begin{proof}
If $M \in \indR$ and $M \notin \C (\subseteq \X_0)$, 
then $\Hom_R(\C,M) = 0$, so $M \in \Y_0$, hence $\X_0 = \add\C$, and the torsion pair $(\X_0, \Y_0)$ is split. 
Since $\C_0 \subseteq \X_0$ and $\rgldim R = 2$, we have $\pd_R Y \leq 1$ for each $Y \in \Y_0$.  
\end{proof}

It remains to show that $R \in \mathcal Y _0$. We do this in several steps. The first one generalizes \cite[Lemma II.1.5]{HRS}:

\begin{lem} \label{lem:one} 
$\mathcal C = \mathcal C_1$.
\end{lem}

\begin{proof}
Suppose there exists $X \in \C_2 \setminus \C_1$. Then
there are the $R$--modules and non--zero $R$--homomorphisms 
$$
\begin{CD} 
Z_0 @>{g_0}>> Y_0 @>{f_0}>> X 
\end{CD}$$
such that $Y_0 \in \C_1 \setminus \C_0$ and $Z_0 \in \C_0$.

We will construct sequences with nonzero maps
$$
\begin{CD} 
Z_i @>{g_i}>> Y_i @>{f_i}>> X 
\end{CD}
$$
such that $Y_i \in \C_1 \setminus \C_0$ and $Z_i \in \C_0$, and $Y_{i+1}$ is either a proper
factor module or a proper submodule of $Y_i$.

\medskip
Assume that the $R$--modules $Y_i, Z_i$ and non--zero $R$--homomorphisms $f_i, g_i$ are defined as above. We proceed by induction on $i$ as follows:

Since $\Hom_R(\C_0,X) = 0$, we have $f_i g_i = 0$, so $\Img g_i \subseteq \Ker f_i$. Note that $X$ and $Y_i$ have projective dimension $\leq 1$, and the same holds for all their submodules.
We distinguish two cases:

\medskip
Case (I) \, $Y_{i + 1} = Y_i/\Img g_i$ is indecomposable.

Since $\Hom_R(Y_{i+1},X) \neq 0$, we have $Y_{i+1} \in \mathcal P_1$. In particular 
$\Ext^2_R(Y_{i+1}, \Ker g_i) = 0$, so the exact sequence 
$$
\begin{CD} 
\mathfrak E : \qquad 0@>>> \Img g_i @>>> Y_i @>>>  Y_{i+1} @>>> 0
\end{CD}
$$
can be obtained by a pushout of an exact sequence of the form
$$
\begin{CD} 
\mathfrak F : \qquad 0 @>>> Z_i @>>> N_i @>{h_i}>>  Y_{i+1} @>>> 0
\end{CD}
$$
along the projection $Z_i \twoheadrightarrow \Img g_i$. Since $Z_i \in \C_0$, we have $\pd_R N_i = 2$ because $\mathcal P _1$ is closed under submodules.

Since $Y_i$ is indecomposable, $\mathfrak E$ is non--split, hence so is $\mathfrak F$. Moreover $Z_i$ is indecomposable,
so there is $Z_{i+1} \in \C_0$ which is a direct summand of $N_i$ and $g_{i+1} = h_i \restriction Z_{i+1} \neq 0$.

There is a unique $f_{i+1} \in \Hom_R(Y_{i+1},X)$ such that $f_{i+1} \pi = f_i$ where $\pi : Y_i \to Y_{i+1}$ is the projection,
and we obtain the $R$--modules and non--zero $R$--homomorphisms 
$$\begin{CD}
Z_{i+1} @>{g_{i+1}}>> Y_{i+1} @>{f_{i+1}}>> X
\end{CD}
$$
with $Y_{i+1} \in \C_1 \setminus \C_0$ and $Z_{i+1} \in \C_0$.

\medskip
Case (II) \, $Y_i/\Img g_i$ is decomposable.

Then there is a decomposition $Y_i/\Img g_i = A_i \oplus B_i$ with 
$A_i = D_i/\Img g_i \in \indR$, $B_i = E_i/\Img g_i \neq 0$, and 
$f_i \restriction D_i \neq 0$. Note that $Y_i = D_i + E_i$ with $D_i\cap E_i = \Img g_i$.

Since $Y_i$ is indecomposable, the sequence $0 \to \Img g_i \to D_i \to A_i \to 0$ is non--split. Let $Y_{i+1}$ be an indecomposable direct summand of $D_i$ \st $f_{i+1} = f_i \restriction Y_{i+1} \neq 0$. Then the composition $h_i: \Img g_i \to Y_{i+1}$ of the inclusion $\Img g_i \to D_i$ with a split projection $D_i \twoheadrightarrow Y_{i+1}$ cannot be zero since $A_i$ is indecomposable. Let $Z_{i+1} = Z_i$ and put $g_{i+1} = h_i g_i$. Then we have the indecomposable $R$--modules and non--zero $R$--homomorphisms 
$$
\begin{CD} 
Z_{i+1} @>{g_{i+1}}>> Y_{i+1} @>{f_{i+1}}>> X
\end{CD}
$$
with $Y_{i+1} \in \C_1 \setminus \C_0$ and $Z_{i+1} \in \mathcal C_0$. Moreover, if $Q_i$ is a complement of $Y_{i+1}$ in $D_i$ (so $Y_i \supsetneq D_i = Y_{i+1} \oplus Q_i$), it is easy to check that exactly one of the following two possibilities occurs:
\begin{enumerate}
\item $Y_{i+1} = D_i$, and so $Y_{i+1} / \Img g_{i+1} = Y_{i+1} / \Img g_i \cong A_i$ is indecomposable. Hence the next step in the construction will be Case (I).
\item $Q_i \ne 0$. Hence we have $0 \ne Q_i \subseteq Y_i$ \st $Q_i \cap Y_{i+1} = 0$.
\end{enumerate}

\medskip
We claim that in the inductive construction above, Case (I) occurs only finitely many times. Indeed, in Case (I), $Y_{i+1}$ is taken as a proper factor (homomorphic image) of $Y_i$ while in Case (II), $Y_{i+1}$ is a proper submodule of $Y_i$. So if Case (I) occurs infinitely many times, the preimages in $Y_0$ of the kernels of the factorizations yield a strictly increasing sequence of submodules of $Y_0$, contradicting its noetherianity.

So without loss of generality, we can assume that only Case (II) occurs. But then we find $0 \ne Q_i \subseteq Y_i$ \st $Q_i \cap Y_{i+1} = 0$ for each $i$, so
$$ Q_0 \subsetneq Q_0 \oplus Q_1 \subsetneq Q_0 \oplus Q_1 \oplus Q_2 \subsetneq \dots $$
is a strictly increasing chain of submodules of $Y_0$, a contradiction.

This proves that $\C_2 \setminus \C_1 = \emptyset$, so $\C = \C_1$.
\end{proof}

\medskip
\begin{lem} \label{cor:char_pd2} $\Hom_R(M,R) = 0$ for each $M \in \mathcal C _0$.
\end{lem}

\begin{proof}
Suppose there exists $M \in \C_0$ with $0\neq f : M \to R$. For $N = f(M)$ and $K = \Ker f$, we get  the following non--split exact sequence in $\modR$:
$$
\begin{CD} 
\mathfrak E : \qquad 0@>>> K @>>> M @>>>  N @>>> 0.
\end{CD}
$$
Since $\rgldim R = 2$, $N$ has projective dimension $\leq 1$, so $\pd_R K = 2$.

From the data $(M,f)$ we will construct $\tilde M\in \C_0$ and  $0\neq \tilde f : \tilde M \to R$ such that all 
indecomposable direct summands of $\tilde K = \Ker \tilde f$ have projective dimension 2:

We have $K = K^\prime \oplus \bar K$ where $\bar K \in \mathcal P_1$ and $0 \neq K^\prime$ has no indecomposable direct summands of projective dimension $\leq 1$. 

Suppose $\bar K \neq 0$. 
Then the pushout of $\mathfrak E$ along the split projection $\rho : K \twoheadrightarrow K^\prime$ yields an exact sequence 
(with $M^\prime$ a proper factor module of $M$)
$$
\begin{CD} 
\mathfrak F : \qquad 0@>>> {K^\prime} @>{\nu}>> {M^\prime} @>{\pi}>>  N @>>> 0.
\end{CD}
$$
This sequence does not split since otherwise $K^\prime$ would be a direct summand of $M$.
Since $0 \neq \bar K \in \mathcal P_1$, we have $\pd_R M^\prime = 2$.
Also $M^\prime = M^{\prime \prime} \oplus \bar M$ where $\bar M \in \mathcal P _1$ and
$0 \neq M^{\prime \prime}$ has no indecomposable direct summands of projective dimension $\leq 1$.
If $M^{\prime \prime} \subseteq \nu(K^\prime)$, then 
$M^{\prime \prime} = \nu(K^\prime)$ (because otherwise $K^\prime$ has a non--zero direct summand isomorphic to a submodule of $\bar M$, hence of projective dimension $\leq 1$), so $\mathfrak F$ splits, a contradiction. This shows that $M^{\prime \prime}$
has an indecomposable direct summand $M_1$ such that $\pd_R M_1 = 2$ and $\pi (M_1) \neq 0$. 
Replacing $N$ by the non--zero submodule $N_1 = \pi (M_1)$, we get a short exact sequence
$$
\begin{CD}
{\mathfrak E}_1 : \qquad 0 @>>> {K_1} @>{{\nu}_1}>> {M_1} @>{\pi_1}>>  N_1 @>>> 0.
\end{CD}
$$
Again $\pd_R K_1 =2$, and $M_1 \in \mathcal C_0$ is a proper factor module of $M$.

Iterating this procedure if necessary, we get short exact sequences $0\to K_i\to M_i\to N_i\to 0$ with
$M_i\in\C_0$, $N_i$ a right ideal in $R$, $\pd K_i =2$ and proper epimorphisms $M_{i-1} \twoheadrightarrow M_i$.

This reduction stops after $i$ steps, if all indecomposable direct summands of $K_i$ have projective dimension 2.
The reduction has to stop, since $M$ is noetherian and $M \twoheadrightarrow M_1 \twoheadrightarrow \cdots$ is a chain of proper epimorphisms.
So we have a non--split short exact sequence 
$$
\begin{CD}
\mathfrak G : \qquad 0 @>>> \tilde K @>>> \tilde M @>>> \tilde N @>>> 0,
\end{CD}
$$
with $\tilde M \in \C_0$, $\tilde N$ a right ideal and $\tilde K$ 
a direct sum of indecomposable modules of projective dimension $2$.

Since $R$ is almost hereditary, $\id_R {\tilde K} \leq 1$. Applying $\Hom_R(-,\tilde K)$ to the exact sequence 
$0 \to \tilde N \to R \to R/{\tilde N} \to 0$, we obtain the exact sequence  
$$
\begin{CD} 
0 = \Ext^1_R(R, \tilde K) @>>> \Ext^1_R(\tilde N, \tilde K) @>>> \Ext^2_R(R/{\tilde N}, \tilde K) = 0.
\end{CD}
$$
So $\Ext^1_R(\tilde N, \tilde K) = 0$ and $\mathfrak G$ splits, a contradiction.
\end{proof}

\begin{proof}[Proof of the Main Theorem]
It only remains to finish the proof of the implication (iii)~$\implies$~(ii); the rest is covered by Corollary~\ref{cor:rings} and Lemma~\ref{lem:back}. To this end, since $\Hom _R(\mathcal C_0,R) = 0$ by the previous lemma, no indecomposable direct summand of $R$ is in $\mathcal C _1 = \mathcal C$. Since the torsion pair $(\mathcal X _0 = \add (\mathcal C),\mathcal Y _0)$ splits, all indecomposable direct summands of $R$, hence also $R$ itself, are contained in $\mathcal Y_0$.
\end{proof}

\begin{rem}\label{compare} Let $(\X,\Y)$ be any split torsion pair such that $\pd \Y \leq 1$. Then $\Y \subseteq \mathcal Y _0$ and $\mathcal X _0 \subseteq \X$: since $\C_0 \cap \Y = \emptyset$, we have $\C_0 \subseteq \X$, and by induction, $\C \subseteq \X$. So $\Y \subseteq \Y_0$.
\end{rem}

\section{Examples}
\label{sec:expl}

We finish by providing examples of right noetherian rings that are almost hereditary, but neither hereditary nor artin algebras.

Generalizing from artin algebras to the right noetherian rings, we normally encounter some classical examples of commutative noetherian rings. It may come as a surprise that these, however, do not fit our setting unless they are hereditary: 

\begin{lem} Let $R$ be a commutative noetherian ring of $\gldim R = 2$. Then $R$ is not almost hereditary.
\end{lem}

\begin{proof} Suppose $R$ is almost hereditary. Then $R$ is a regular ring of Krull dimension $2$, so $\id R = 2$, and there is a prime ideal $q$ of height $2$. By Bass' Theorems \cite[18.7 and 18.8]{Ma}, the localization $R_q$ is also regular of Krull dimension $2$, and the Bass invariant 
$\mu_2(q,R) = \dim_{k(q)}\Ext ^2_{R_q}(k(q),R_q) = \dim_{k(q)}\Ext ^2_R(R/q,R)_q = 1$ where $k(q) = R_q/q_q$ is the residue field. In particular, $R/q$ is an indecomposable module of projective dimension $2$, so $\id R/q \leq 1$ since $R$ is almost hereditary. Then also $\id_{R_q} k(q) \leq 1$ by \cite[Lemmas 5 and 6]{Ma}. This contradicts the equality $\id_{R_q} k(q) = \depth R_q = \dim R_q = 2$ of \cite[9.2.17]{EJ}.
\end{proof}

Fortunately, a class of non--commutative noetherian examples can be obtained by applying some of the results of Colby and Fuller \cite{CbF}.
If $S$ is a left and right noetherian serial ring and $T \in \modS$ is a tilting module then the ring $R = \End_S(T)$ is called 
\emph{serially tilted} (from $S$). By \cite[\S3]{CbF}, serially tilted rings are semiperfect and noetherian. Non--artinian indecomposable serially tilted rings that are not serial were characterized in \cite[\S4]{CbF} as the rings $R$ from the following example:

\begin{expl} \label{expl:serial}
Let $n$ be a positive integer, $D$ be a local noetherian non--artinian serial ring with the radical $M$ and the associated division ring $C = D/M$, and let $S = UTM_n(D:M)$ denote the subring of the full matrix ring $M_n(D)$ consisting of those matrices all of whose entries below the main diagonal are in $M$. Then $S$ is a noetherian hereditary semiperfect serial ring. (In fact each indecomposable noetherian semiperfect serial ring which is not artinian is Morita equivalent to such $S$ by~\cite[Theorem 5.14]{W}, see also~\cite[Theorem 6.1]{M}.)

The structure of $\modS$ is rather completely described in~\cite[Appendix B]{CbF2}. Namely, any finitely generated module over $S$ decomposes into a direct sum of uniserial modules, which are either projective or of finite length. Therefore, any indecomposable finitely generated module has a local endomorphism ring and the decompositions into indecomposables are unique in the sense of the Krull-Schmidt theorem. It follows easily that the non-isomorphisms between indecomposable modules generate the unique maximal two-sided ideal of $\modS$ containing no non-zero identity morphisms. This ideal, which we call $\radS$, is nothing else then the Jacobson radical of $\modS$, that is, the intersection of all left (or right) maximal ideals of $\modS$.

At this point, there are many similarities with representation theory of artin algebras. Each indecomposable module $X \in \modS$ admits a minimal right almost split morphism $f: E \to X$ in the sense of~\cite[\S V.1]{ARS} and $\radS$ is generated by irreducible morphisms (see~\cite[\S V.5]{ARS}) as a right ideal of $\modS$. Moreover, one can draw the Auslander-Reiten quiver of $S$ with isomorphism classes of indecomposable $S$-modules as vertices and arrows whenever there exists an irreducible morphism:
$$
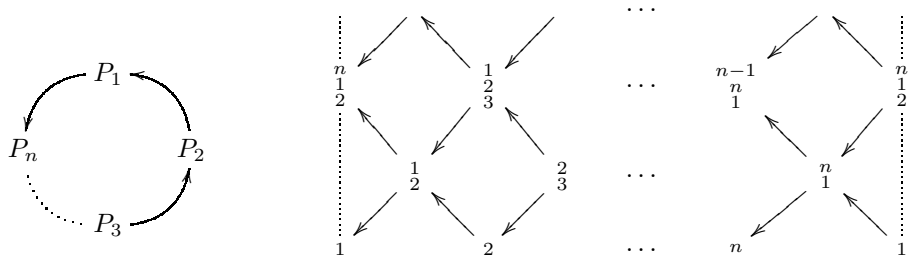
\xymatrix@=5mm{
\\
&
P_1 \ar@/_.9pc/[dl] &
\\
P_n \ar@/_.9pc/@{..}[dr] &&
P_2 \ar@/_.9pc/[ul]
\\
&
P_3 \ar@/_.9pc/[ur] &
}
\qquad \qquad
\xymatrix@=5mm{
\ar@{..}[d] &
\ar[dl] &&
\ar[dl] &
\dots &&
\ar[dl] &
\ar@{..}[d]
\\
\save[] {\begin{smallmatrix} n \\ 1 \\ 2 \end{smallmatrix}} \restore \ar@{..}[dd] && 
\save[] {\begin{smallmatrix} 1 \\ 2 \\ 3 \end{smallmatrix}} \restore \ar[dl] \ar[ul] &&
\dots &
\save[] {\begin{smallmatrix} n-1 \\ n \\ 1 \end{smallmatrix}} \restore &&
\save[] {\begin{smallmatrix} n   \\ 1 \\ 2 \end{smallmatrix}} \restore \ar[dl]  \ar[ul] \ar@{..}[dd]
\\
&
\save[] {\begin{smallmatrix}1 \\ 2\end{smallmatrix}} \restore \ar[dl] \ar[ul] &&
\save[] {\begin{smallmatrix}2 \\ 3\end{smallmatrix}} \restore \ar[dl] \ar[ul] &
\dots &&
\save[] {\begin{smallmatrix}n \\ 1\end{smallmatrix}} \restore \ar[dl] \ar[ul] &
\\
\save[] {\begin{smallmatrix} 1 \end{smallmatrix}} \restore && 
\save[] {\begin{smallmatrix} 2 \end{smallmatrix}} \restore \ar[ul] &&
\dots &
\save[] {\begin{smallmatrix} n \end{smallmatrix}} \restore &&
\save[] {\begin{smallmatrix} 1 \end{smallmatrix}} \restore \ar[ul]
\\
}
$$
There are, however, substantial differences from artin algebras, too. The indecomposable projectives do not admit any left almost split morphisms in $\modS$ and they form a cycle in the Auslander-Reiten quiver.

Let now $l$, $m$ and $m_j \, (j = 1,\dots,l)$ be positive integers \st $l \le m$ and $n = m + \sum_{j = 1}^l m_j$.
Let $A = UTM_m(D:M)$, and for each $j = 1,\dots,l$, let $A_j = UTM_{m_j}(C)$ be the upper triangular matrix ring of degree $m_j$ over $C$. Let further $T_j = T^\prime_j \oplus U_j$ be a basic tilting right $A_j$--module with $U_j$ the minimal faithful (= indecomposable projective injective) $A_j$--module. Let $B_j = \End_{A_j}(T_j)$, and $B = B_1 \times \dots \times B_l$ (the ring direct product).

Finally, let ${_AX}$ be semisimple, with orthogonal idempotents $g_1,\dots,g_l \in A$ such that $X_B \cong \bigoplus_{j} g_jX_B$, and $g_jX_B \cong \Hom_{A_j}(U_j, T_j)$ as $C$--$B$--bimodules for each $j = 1,\dots,l$. By \cite[Theorem 4.5]{CbF}, the ring 
$$
R = 
\begin{pmatrix} 
A&X\\0&B
\end{pmatrix}
$$
is serially tilted (from $S$), and $R$ is indecomposable, but neither serial nor artinian. Moreover, each serially tilted ring with the latter properties is isomorphic to some $R$ as above. Since $R$ is not serial, it is not right hereditary.
\end{expl}

\medskip
By the Main Theorem, the rings $R$ from Example~\ref{expl:serial} yield the desired examples of non--artinian non--hereditary almost hereditary rings. 

\bigskip
In \cite{HRS}, for any artin algebra $R$, two classes of indecomposable modules, $\mathcal L$ and $\mathcal R$, were defined as follows: 
$$\mathcal L = \{ M \in \indR \mid \pd N \leq 1 \hbox{ for all } N \rightsquigarrow M \}$$ 
and 
$$\mathcal R = \{ M \in \indR \mid \id N \leq 1 \hbox{ for all } M \rightsquigarrow N \}$$ 
where $X \rightsquigarrow Y$ means that there is a finite sequence of indecomposable modules $X = X_0, X_1, \dots, X_s = Y$ such that 
$\Hom _R(X_i,X_{i+1}) \neq 0$ for each $i < s$.     

In \cite[p.36]{HRS} and \cite[p.61]{H2}, the question of whether \emph{always} $\mathcal L \cap \mathcal R \neq \emptyset$ was raised as the main open problem for quasi--tilted artin algebras; a positive answer was obtained by Happel in 2000 (see \cite[Corollary 2.8]{H3}). In the next example we will see that in our general setting of quasi--tilted noetherian rings, a negative answer is possible even for serially tilted rings. So unlike Section \ref{sec:q_tilt} which as byproduct gives a simpler module--theoretic proof even in the artin algebra case, our approach does not yield any module--theoretic proof of $\mathcal L \cap \mathcal R \neq \emptyset$ for artin algebras.     

\begin{expl} \label{expl:empty}
Let $p$ be a prime integer, $\mathbb Z _p$ the field with $p$ elements, and $\mathbb Z _{(p)}$ the localization of $\mathbb Z$ at $p \mathbb Z$. Let 
$$
R = 
\begin{pmatrix} 
{\mathbb Z _p}&{\mathbb Z _p}\\0&{\mathbb Z _{(p)}}
\end{pmatrix}
$$
By \cite[\S4]{CbF}, $R$ is serially tilted from the ring 
$$
H = 
\begin{pmatrix} 
{\mathbb Z _{(p)}}&{\mathbb Z _{(p)}}\\{p \mathbb Z _{(p)}}&{\mathbb Z _{(p)}}
\end{pmatrix}
$$
Indeed, for $e_1 = \left (\smallmatrix 1&0\\0&0 \endsmallmatrix \right ) \in H$, $e_2 = \left (\smallmatrix 0&0\\0&1 \endsmallmatrix \right ) \in H$, and $P_i = e_iH$, one has the short exact sequence $0 \to P_2 \to P_1 \to S_1 \to 0$ with $S_1$ simple. Using this short exact sequence it is easy to see that $T = P_1 \oplus S_1$ is a finitely generated tilting $H$--module with $\End _H(T) \cong R$.
This shows that $R$ is right noetherian, almost hereditary, but not hereditary, and not artinian.
  
Let $e = \left (\smallmatrix 1&0\\0&0 \endsmallmatrix \right ) \in R$, $f = \left (\smallmatrix 0&0\\0&1 \endsmallmatrix \right ) \in R$, and 
$g = \left (\smallmatrix 0&1\\0&0 \endsmallmatrix \right ) \in R$. 

Note that arbitrary right $R$--modules $M$ can be identified with the triples $(L,N,\varphi)$ where $L$ is a linear space over $\mathbb Z _p$, $N$ is 
a $\mathbb Z _{(p)}$--module, and $\varphi : L \to \mbox{Soc}(N)$ is a $\mathbb Z_p$--linear map (in fact, $L = Me$, $N = Mf$, and $\varphi$ is induced by the multiplication by $g$; for short, we shall not distinguish between $\varphi$ and the corresponding $\mathbb Z_{(p)}$--linear map from 
$L$ to $N$). $R$--homomorphisms then correspond to the pairs $(\alpha, \beta)$ where $\alpha$ is $\mathbb Z _p$--linear, and $\beta$ is a $\mathbb Z_{(p)}$--homomorphism and the obvious diagram commutes (see e.g.\ \cite[III.2]{ARS}). 

Note that the simple module $S = eR/gR$ corresponds to the triple $(\mathbb Z_p, 0, 0)$, so an embedding of $S$ into any module splits, and $S$ is injective.

We claim that for each module $M$, $\pd_R M = 2$ \iff $M$ contains a direct summand isomorphic to $S$. The if--part is clear since $S$ has projective dimension $2$. Conversely, let $M$ be with $\pd_R M = 2$ and let $(L,N,\varphi)$ be the corresponding triple. If $N = 0$ then $M = Me$ is semisimple, and contains $S$.

Assume $N \neq 0$. If $\varphi$ is not monic, then $S$ embeds into $M$, hence is its direct summand, because $S$ is injective. 
Assume the map $\varphi$ is monic. Let $M^\prime$ be the submodule of $M$ corresponding to the triple $(L,\mbox{Im}(\varphi),\varphi)$. Then $M^\prime$ is isomorphic to a direct sum of copies of the module $eR = (\mathbb Z_p,\mathbb Z_p,\mbox{id})$; in particular, $M^\prime$ is projective, so the module $\bar M = M/M^\prime$ has projective dimension $2$. However, $\bar M = \bar M f$, so $\bar M$ has projective dimension $\leq 1$, a contradiction. This proves our claim.

Next, we describe the elements of $\indR$. By the above, $M \in \indR$ has projective dimension $2$ \iff $M \cong S$ \iff $M$ is simple and injective. 

If $M = (L,N,\varphi) \in \indR$ has projective dimension $\leq 1$, then $\varphi$ is monic, and $N = N_t \oplus N_f$ where $N_t$ is torsion and $N_f$ is free (as $fRf$--modules). Since $\mbox{Soc}(N_f) = 0$, this yields a decomposition $M = (L,N_t,\varphi) \oplus (0,N_f,0)$ in $\modR$. 
Hence either $M \cong (0,N_f,0) \cong fR$ is projective, or $N_f = 0$. In the latter case, there are two possibilities: 

1.\ $L = 0$. Then $M \cong (0,\mathbb Z _{p^r},0)$ for some $r \geq 1$. 

2.\ $L \neq 0$. Then $M \cong (\mathbb Z _p ,\mathbb Z _{p^s},\varphi)$ for some $s \geq 1$. This follows from the well known fact that the cyclic group generated by any element of maximal order in an abelian $p$--group splits off.      

Note that all indecomposable modules $M$ non--isomorphic to $S$ have injective dimension $2$, because $\Ext ^1_R(gR,M) \neq 0$, so $\Ext ^2_R(R/gR,M) \neq 0$. 

Now it is easy to compute the classes $\mathcal L$ and $\mathcal R$ in our setting: $\mathcal L = \indR \setminus \{ S \}$, and  $\mathcal R = \{ S \}$,
so clearly $\mathcal L \cap \mathcal R = \emptyset$.

Finally, note that there is only one split torsion pair $(\mathcal X, \mathcal Y)$ in $\modR$ such that $\mathcal Y$ consists of modules of projective dimension $\leq 1$ and $R \in \mathcal Y$, namely $(\mathcal X _0,\mathcal Y_0)$ (see Remark \ref{compare}). Here $\mathcal X _0 = \add (S)$ and $\mathcal Y _0 = \add (\mathcal L)$. 

Indeed, let $M \in \mathcal X \setminus \mathcal X _0$. W.l.o.g.\ $M \in \indR$, so by the classification of $\indR$ given above either $M \cong fR$ (which contradicts $R \in \mathcal Y$) or $M$ has a factor--module isomorphic to $gR$, so $eR \in \mathcal X$ because of the exact sequence 
$0 \to gR \to eR \to S \to 0$ (which again contradicts $R \in \mathcal Y$).    
\end{expl}


\end{document}